\documentclass[12pt]{amsart}
\usepackage{amscd, amsfonts, amssymb}
\usepackage{fullpage}
\usepackage{latexsym}
\usepackage{verbatim}
\usepackage{enumerate}
\usepackage[colorlinks=true,citecolor=red,linkcolor=blue]{hyperref}
\usepackage{tikz}
\usetikzlibrary{arrows,matrix,decorations,positioning, backgrounds,fit}

\renewcommand\gg{\mathfrak g}

\newcommand\hh{\mathfrak h}

\newcommand\inverse{{^{-1}}}

\newcommand\ra{\rightarrow}

\newcommand{\KK}{{\mathbb K}}
\newcommand{\NN}{{\mathbb N}}
\newcommand{\ZZ}{{\mathbb Z}}
\newcommand{\QQ}{{\mathbb Q}}
\newcommand{\RR}{{\mathbb R}}

\newcommand{\tuple}[1]{{\mathbf {#1}}}

\DeclareMathOperator{\Gal}{Gal}

\DeclareMathOperator{\Lie}{Lie}

\DeclareMathOperator{\IM}{Im}

\numberwithin{equation}{section}

\newtheorem{thm}[equation]{Theorem}

\newtheorem{lem}[equation]{Lemma}

\newtheorem{prop}[equation]{Proposition}

\theoremstyle{definition}
\newtheorem{defn}[equation]{Definition}
\theoremstyle{remark}
\newtheorem{rem}[equation]{Remark}
\theoremstyle{remark}

\newcommand{\ovl}{\overline}


\title[Cocharacter-closure and spherical buildings]
{Cocharacter-closure and spherical buildings}

\author[M.\  Bate]{Michael Bate}
\address%[M.\  Bate]
{Department of Mathematics,
University of York,
York YO10 5DD,
United Kingdom}
\email{michael.bate@york.ac.uk}

\author[S. Herpel]{Sebastian Herpel}
\address%[S. Herpel]
{Fakult\"at f\"ur Mathematik,
Ruhr-Universit\"at Bochum,
D-44780 Bochum, Germany}
\email{sebastian.herpel@rub.de}

\author[B.\ Martin]{Benjamin Martin}
\address%[B.\ Martin]
{Department of Mathematics,
University of Aberdeen,
King's College,
Fraser Noble Building,
Aberdeen AB24 3UE,
United Kingdom}
\email{b.martin@abdn.ac.uk}

\author[G. R\"ohrle]{Gerhard R\"ohrle}
\address%[G.~R\"{o}hrle]
{Fakult\"at f\"ur Mathematik,
Ruhr-Universit\"at Bochum,
D-44780 Bochum, Germany}
\email{gerhard.roehrle@rub.de}

\dedicatory{In memory of Robert Steinberg}

\begin{document}

\begin{abstract}
Let $k$ be a field, let $G$ be a reductive $k$-group and
$V$ an affine $k$-variety on which $G$ acts.
In this note we continue our study of 
the notion of cocharacter-closed $G(k)$-orbits in $V$.
In earlier work we used 
a rationality condition on the point stabilizer of a $G$-orbit to 
prove Galois ascent/descent and Levi ascent/descent 
results concerning cocharacter-closure
for the corresponding $G(k)$-orbit in $V$.
In the present paper 
we employ building-theoretic techniques to derive
analogous results.
\end{abstract}

\maketitle

\setcounter{tocdepth}{1}
%\tableofcontents

\section{Introduction}
\label{sec:intro}

Let $k$ be a field and let $G$ be a reductive linear algebraic group acting
on an affine variety $V$, with $G$, $V$ and the action all defined over $k$.  Let $\Delta_k$ be the (simplicial) spherical building of $G$ over $k$, and let $\Delta_k(\RR)$ be its geometric realisation (for precise definitions, see below).  In this paper we continue the study, initiated in \cite{GIT}, \cite{stcc} and \cite{BHMR:cochars},
of the notion of \emph{cocharacter-closed orbits} in $V$ for the group $G(k)$ of $k$-rational points of $G$, and of interactions with the geometry of $\Delta_k(\RR)$.  The philosophy of this paper is as follows (cf.\ \cite{BHMR:cochars}): for a point $v$ in $V$, we are interested in \emph{Galois ascent/descent} questions---given a separable algebraic extension $k'/k$ of fields, how is the $G(k')$-orbit of $v$ related to the $G(k)$-orbit of $v$?---and \emph{Levi ascent/descent} questions---given a $k$-defined torus $S$ of the stabilizer $G_v$, how is the $C_G(S)(k)$-orbit of $v$ related to the $G(k)$-orbit of $v$?  (See \cite[Sec.~5, para.~1]{BHMR:cochars} for an explanation of the terms Galois/Levi ascent/descent in this context.) 
These questions are related, and have natural interpretations in
$\Delta(\KK)$.

Our results complement those of \cite{BHMR:cochars}: they give similar conclusions but under different assumptions.  It was shown in \cite{BHMR:cochars} (see Proposition \ref{prop:galois_descent} below) that Galois descent---passing from $G(k')$-orbits to $G(k)$-orbit---is always well-behaved.
Certain results on Galois ascent were also proved \cite[Thm.\ 5.7]{BHMR:cochars} under hypotheses
on the stabilizer $G_v$.  The mantra in this paper is that when the Centre Conjecture (see Theorem~\ref{thm:centreconj} below) is known to hold, one can use it to prove Galois ascent results, and hence deduce Levi ascent/descent results.  The idea is that when the extension $k'/k$ is separable and normal, questions of Galois ascent can be interpreted in terms of the action of the Galois group of $k'/k$ on the building; moreover, if one has such Galois ascent questions under control, then it is easier to handle Levi ascent/descent because one may assume that the torus $S$ is split (cf.\ \cite[Thm.\ 5.4(ii)]{BHMR:cochars}).

When $k$ is algebraically closed (or more generally when $k$ is perfect),
our setup is also intimately related to the \emph{optimality formalism} of Kempf-Rousseau-Hesselink,
\cite{kempf}, \cite{rousseau}, \cite{He}.  Indeed, one may interpret this formalism in the language of the Centre Conjecture (see \cite[Sec.\ 1]{stcc}).  The idea is to study the $G$-orbits in $V$ via limits along cocharacters of $G$: limits are formally defined below, but given $v$ in $V$, if we take the
set of cocharacters $\lambda$ of $G$ for which the limit $\lim_{a\to 0}\lambda(a)\cdot v$ exists,
and interpret this set in terms of the set of $\QQ$-points $\Delta(\QQ)$ of the building of $G$,
then we obtain a convex subset $\Sigma_v$ of $\Delta(\QQ)$.
In case $G\cdot v$ is not Zariski-closed, one can find a fixed point in the set $\Sigma_v$
and an associated \emph{optimal parabolic subgroup} $P$ of $G$ with many nice properties:
in particular, the stabilizer $G_v$ normalizes $P$.
It is not currently known in general how to produce analogues of these optimality
results over arbitrary fields (or even whether such results exist); see \cite[Sec.\ 1]{GIT} for further discussion.
Our first main theorem gives a rational analogue of the Kempf-Rousseau-Hesselink
ideas when $\Sigma_{v,k_s}$ (the points of $\Sigma_v$ coming from $k_s$-defined
cocharacters of $G$) happens to be a subcomplex of $\Delta(\QQ)$,
and also answers in this case the ascent/descent questions posed earlier.

\begin{thm}
\label{thm:subcx}
Let $v \in V$. Suppose $\Sigma_{v,k_s}$ is a subcomplex of $\Delta_{k_s}(\QQ)$.
Then the following hold:
 \begin{itemize}
  \item[(i)] Suppose $v\in V(k)$ and
  $G(k_s)\cdot v$ is not cocharacter-closed over $k_s$.  Let $S$ be any 
  $k$-defined torus of $G_v$ and set $L= C_G(S)$.  Then
there exists $\sigma\in Y_k(L)$ such that 
  $\lim_{a\to 0} \sigma(a) \cdot v$ exists and lies outside $G(k_s) \cdot v$.
  \item[(ii)] Suppose $v \in V(k)$. For any separable algebraic extension
  $k'/k$, $G(k')\cdot v$ is cocharacter-closed over $k'$ if and only if
  $G(k) \cdot v$ is cocharacter-closed over $k$.
  \item[(iii)] Let $S$ be any $k$-defined torus of $G_v$ and set $L= C_G(S)$.  Then $G(k)\cdot v$ is cocharacter-closed over $k$ if and only if $L(k)\cdot v$ is cocharacter-closed over $k$.
\end{itemize}
\end{thm}

The hypothesis that $\Sigma_{v,k_s}$ is a subcomplex allows us to apply the following result---Tits' Centre Conjecture---in the proof of Theorem~\ref{thm:subcx}:

\begin{thm}
\label{thm:centreconj}
 Let $\Theta$ be a thick spherical building and let $\Sigma$ be a convex subcomplex of $\Theta$ such that $\Sigma$ is not completely reducible.  Then there is a simplex of $\Sigma$ that is fixed by every building automorphism of $\Theta$ that stabilizes $\Sigma$.  (We call such a simplex a {\em centre} of $\Sigma$.)
\end{thm}

\noindent For definitions and further details, see \cite{rc}; in particular, note that the spherical building of a reductive algebraic group is thick.  The conjecture was proved by M\"uhlherr-Tits \cite{muhlherr-tits}, Leeb-Ramos-Cuevas \cite{lrc} and Ramos-Cuevas \cite{rc}, and a uniform proof for chamber subcomplexes has also now been given by M\"uhlherr-Weiss \cite{muhlherr-weiss}.  The condition that $\Sigma_{v,k_s}$ is a subcomplex is satisfied in the theory of complete reducibility for subgroups of $G$ and Lie subalgebras of $\Lie(G)$, and our results yield applications to complete reducibility (see Theorem~\ref{thm:LeviGaloisdescent-ascentforGcr} below).

It was shown in \cite[Thm.\ 5.7]{BHMR:cochars} that the conclusions of Theorem \ref{thm:subcx}(ii) and (iii) hold if $G_v$ has a maximal torus that is $k$-defined.  Our second main result gives alternative hypotheses on $G_v$, this time of a group-theoretic nature, for the conclusions of Theorem \ref{thm:subcx} to hold, without the assumption that $\Sigma_{v,k_s}$ is a subcomplex.  The proof of this result relies in an essential way on known cases of a strengthened version of the Centre Conjecture
(this time from \cite{stcc}).

\begin{thm}
\label{thm:nice_hyps}
Let $v\in V(k)$.  Suppose that
(a) $G_v^0$ is nilpotent, or
(b) every simple component of $G^0$ has rank 1.  Then the following hold:
 \begin{itemize}
  \item[(i)] Suppose $G(k_s)\cdot v$ is not cocharacter-closed over $k_s$.
  Let $S$ be any $k$-defined torus of $G_v$ and set $L= C_G(S)$.  Then 
there exists $\sigma\in Y_k(L)$ such that $G_v(k_s)$ normalizes $P_\sigma(G^0)$ and
  $\lim_{a\to 0} \sigma(a) \cdot v$ exists and lies outside $G(k_s) \cdot v$.
  \item[(ii)] For any separable algebraic extension
  $k'/k$, $G(k')\cdot v$ is cocharacter-closed over $k'$ if and only if
  $G(k) \cdot v$ is cocharacter-closed over $k$.
  \item[(iii)] Let $S$ be any $k$-defined torus of $G_v$ and set $L= C_G(S)$.  Then 
$G(k)\cdot v$ is cocharacter-closed over $k$ if and only if $L(k)\cdot v$ is cocharacter-closed over $k$.
\end{itemize}
\end{thm}

The hypothesis in Theorem \ref{thm:subcx}(i) that $v$ is a $k$-point ensures that
the subset $\Sigma_v$ is Galois-stable, and it is also needed in our proof of Theorem~\ref{thm:nice_hyps} (but see Remark~\ref{rem:arbpt}).
Sometimes, however, one can get away with a weaker hypothesis.
This happens for $G$-complete reducibility
in the final section of the paper, where we prove the following ascent/descent result:

\begin{thm}
\label{thm:LeviGaloisdescent-ascentforGcr}
Suppose that $G$ is connected.
Let $H$ be a subgroup of $G$.
Let $S$ be a $k$-defined torus of $C_G(H)$ and set $L= C_G(S)$.
Then $H$ is $G$-completely reducible over $k$ if and only if
$H$ is $L$-completely reducible over $k$.
\end{thm}

\begin{rem}
\label{rem:LeviGaloisdescent-ascentforGcr}
(i). Theorem \ref{thm:LeviGaloisdescent-ascentforGcr}
gives an alternative proof and also slightly generalizes
Serre's Levi ascent/descent result
\cite[Prop.\ 3.2]{serre1.5}; cf.\ \cite[Cor.\ 3.21, Cor.\ 3.22]{BMR}.
For in the statement of Theorem \ref{thm:LeviGaloisdescent-ascentforGcr},
we do not require $H$ to be a subgroup of $G(k)$.

(ii). The counterpart of Theorem \ref{thm:subcx}(ii) (Galois ascent/descent for $G$-complete reducibility) was proved in \cite{BMR:tits}.
\end{rem}

We spend much of the paper recalling relevant results from geometric invariant theory and the theory of buildings.  Although the basic ideas are familiar, we need to extend many of them: for instance, the material on quasi-states in Section~\ref{sec:stcc} was covered in \cite{stcc} for algebraically closed fields, but we need it for arbitrary fields.  We work with the geometric realisations of buildings rather than with buildings as abstract simplicial complexes; some care is needed when the reductive group $G$ has positive-dimensional centre.

The paper is laid out as follows. In Section \ref{sec:prelims}, we set up notation and collect
terminology and results relating to cocharacter-closedness.
In Section \ref{sec:TCC}, we translate our setup into the language of spherical buildings; we use notation and results from \cite{stcc} on buildings, some of which we extend slightly.
In Section \ref{sec:tccproofs}, we combine the technology from both of the preceding sections to give proofs of our main results.
In the final section we give our applications to the theory of complete reducibility.

\section{Notation and preliminaries}
\label{sec:prelims}

Let $k$ denote a field with separable closure $k_s$ and algebraic closure $\ovl{k}$.
Let $\Gamma:=\Gal(k_s/k) = \Gal(\ovl{k}/k)$ denote the Galois group of $k_s/k$.
Throughout, $G$ denotes a (possibly non-connected) reductive linear algebraic group defined
over $k$,
and $V$ denotes a $k$-defined affine variety upon which $G$ acts $k$-morphically.
Let $G(k)$, $G(k_s)$, $V(k)$, $V(k_s)$ denote the $k$- and $k_s$-points of $G$ and $V$;
we usually identify $G$ with $G(\ovl{k})$ and $V$ with $V(\ovl{k})$.  If $X$ is a variety then we denote its Zariski closure by $\ovl{X}$.

More generally, we need to consider $k$-points and $k_s$-points in subgroups that are not necessarily $k$-defined or $k_s$-defined; note that if $k$ is not perfect then even when $v$ is a $k$-point, the stabilizer $G_v$ need not be $k$-defined.  If $k'/k$ is an algebraic field extension and $H$ is a closed subgroup of $G$ then we set $H(k')= H(\ovl{k})\cap G(k')$, and we say that a torus $S$ of $H$ is $k'$-defined if it is $k'$-defined as a torus of the $k$-defined group $G$.  Note that a $k_s$-defined torus of $H$ is a torus of $\ovl{H(k_s)}$.

\subsection{Cocharacters and $G$-actions}

Given a $k$-defined algebraic group $H$, we let $Y(H)$ denote the set of cocharacters of $H$,
with $Y_{k}(H)$ and $Y_{k_s}(H)$ denoting the sets of $k$-defined and $k_s$-defined cocharacters, respectively.
The group $H$ acts on $Y(H)$ via the conjugation action of $H$ on itself.
This gives actions of the group of $k$-points $H(k)$ on $Y_{k}(H)$ and the group of $k_s$-points $H(k_s)$ on $Y_{k_s}(H)$.
There is also an action of the Galois group $\Gamma$ on $Y(H)$ which stabilizes $Y_{k_s}(H)$,
and the $\Gamma$-fixed elements of $Y_{k_s}(H)$ are precisely the elements of $Y_{k}(H)$.
We write $Y = Y(G)$, $Y_k = Y_{k}(G)$ and $Y_{k_s} = Y_{k_s}(G)$.

\begin{defn}\label{defn:norm}
A function $\left\|\,\right\|:Y \to \RR_{\geq0}$ is called a \emph{$\Gamma$-invariant, $G$-invariant norm} if:
\begin{itemize}
\item[(i)] $\left\| g\cdot\lambda\right\| =\left\|\lambda\right\|=\left\|\gamma\cdot\lambda\right\|$ for all
$\lambda \in Y$, $g\in G$ and $\gamma \in \Gamma$;
\item[(ii)] for any maximal torus $T$ of $G$, there is a positive definite integer-valued form
$(\ ,\ )$ on $Y(T)$ such that $(\lambda,\lambda) = \left\|\lambda\right\|^2$ for any $\lambda\in Y(T)$.
\end{itemize}
\end{defn}

Such a norm always exists.  For take a $k$-defined maximal torus
$T$ and any positive definite integer-valued form on $Y(T)$.
Since $T$ splits over a finite extension of $k$, we can average the form over the Weyl group $W$
and over the finite Galois group of the extension to obtain a $W$-invariant $\Gamma$-invariant form
on $Y(T)$, which defines a norm satisfying (ii).
One can extend this norm to all of $Y$ because any cocharacter is $G$-conjugate to one in $Y(T)$;
this procedure is well-defined since the norm on $Y(T)$ is $W$-invariant.
See \cite{kempf} for more details.  If $G$ is simple then $\left\|\,\right\|$ is unique up to nonzero scalar multiples.  We fix such a norm once and for all.

For each cocharacter $\lambda \in Y$ and each $v\in V$, we
define a morphism of varieties
$\phi_{v,\lambda}:\ovl{k}^* \to V$ via the formula
$\phi_{v,\lambda}(a) = \lambda(a)\cdot v$.
If this morphism extends to a morphism
$\widehat\phi_{v,\lambda}:\ovl{k} \to V$, then
we say that $\underset{a\to 0}{\lim} \lambda(a) \cdot v$ exists,
and set this limit
equal to $\widehat\phi_{v,\lambda}(0)$; note that such an extension,
if it exists, is necessarily unique.

\begin{defn}
For $\lambda \in Y$ and $v\in V$, we say that
\emph{$\lambda$ destabilizes $v$} provided
$\underset{a\to 0}{\lim} \lambda(a) \cdot v$ exists, and if
$\underset{a\to 0}{\lim} \lambda(a) \cdot v$ exists
and does not belong to $G\cdot v$, then we say \emph{$\lambda$ properly destabilizes $v$}.
We have an analogous notion over $k$: 
if $\lambda \in Y_k$ then we say that $\lambda$ \emph{properly destabilizes $v$ over $k$} if $\underset{a\to 0}{\lim}\lambda(a)\cdot v$ exists and does not belong to $G(k)\cdot v$. 
Finally, if $k'/k$ is an algebraic extension, and $\lambda \in Y_k$, then we say that 
$\lambda$ \emph{properly destabilizes $v$ over $k'$} if 
$\underset{a\to 0}{\lim}\lambda(a)\cdot v$ exists and does not belong to $G(k')\cdot v$;
that is, if $\lambda$ -- regarded as an element of $Y_{k'}(G)$ -- properly destabilizes $v$ over $k'$.
\end{defn}

\subsection{R-parabolic subgroups}

When $V=G$ and $G$ is acting by conjugation,
for each $\lambda \in Y$ we get a set
$P_\lambda:=\{g\in G\mid \lim_{a\to0} \lambda(a)g\lambda(a)\inverse \textrm{ exists}\}$;
this is a parabolic subgroup of $G$.
We distinguish these parabolic subgroups
by calling them \emph{Richardson-parabolic} or \emph{R-parabolic} subgroups.
For basic properties of these subgroups, see \cite[Sec.\ 6]{BMR}.
We recall here that $L_\lambda = C_G(\IM(\lambda))$ is called an \emph{R-Levi subgroup} of $P_\lambda$,
$R_u(P_\lambda)$ is the set of elements sent to $1\in G$ in the limit,
and $P_\lambda = R_u(P_\lambda) \rtimes L_\lambda$.
Further, $R_u(P_\lambda)$ acts simply transitively on the set of all $L_\mu$ such that $P_\mu = P_\lambda$
(that is, on the set of all R-Levi subgroups of $P_\lambda$):
note that this is a transitive action of $R_u(P_\lambda)$
on the set of \emph{subgroups} of the form $L_\mu$, not on the set of
cocharacters for which $P_\mu = P_\lambda$.
Most of these things work equally well over the field $k$: for example, if $\lambda$ is
$k$-defined then $P_\lambda$, $L_\lambda$ and $R_u(P_\lambda)$ are;
moreover, given any $k$-defined R-parabolic subgroup $P$,
$R_u(P)(k)$ acts simply transitively on the set of $k$-defined R-Levi subgroups of $P$
\cite[Lem.\ 2.5]{GIT}.  Note that if $P$ is $k$-defined and $G$ is connected then $P= P_\lambda$ for some $k$-defined $\lambda$, but this can fail if $G$ is not connected \cite[Sec.\ 2]{GIT}.

When $H$ is a reductive subgroup of $G$ the inclusion $Y(H)\subseteq Y(G)$ means that we
get an R-parabolic subgroup of $H$ and of $G$ attached to any $\lambda\in Y(H)$.
When we use the notation $P_\lambda$, $L_\lambda$, etc., we are always thinking of
$\lambda$ as a cocharacter of $G$.
If we need to restrict attention to the subgroup $H$ for some reason,
we write $P_\lambda(H)$, $L_\lambda(H)$, etc.

\subsection{The sets $Y(\QQ)$ and $Y(\RR)$}
\label{subsec:Y}

Form the set $Y(\QQ)$ by taking the quotient of $Y\times \NN_0$ by the relation
$\lambda \sim \mu$ if and only if $n\lambda = m\mu$ for some $m,n \in \NN$, and
extend the norm function to $Y(\QQ)$ in the obvious way.
For any torus $T$ in $G$ , $Y(T, \QQ) := Y(T) \otimes_\ZZ \QQ$ is a vector space over $\QQ$.
Now one can form real spaces $Y(T, \RR) := Y(T, \QQ) \otimes_\QQ\RR$ for each
maximal torus $T$ of $G$ and a set $Y(\RR)$ by glueing the $Y(T, \RR)$ together according to the way the spaces $Y(T, \QQ)$ fit together \cite[Sec.\ 2.2]{stcc}.
The norm extends to these sets.
One can define sets $Y_k(\QQ)$, $Y_k(\RR)$, $Y_{k}(T, \QQ)$, $Y_{k}(T, \RR)$,
 $Y_{k_s}(\QQ)$, $Y_{k_s}(\RR)$, $Y_{k_s}(T, \QQ)$ and $Y_{k_s}(T, \RR)$ analogously by restricting
attention to $k$-defined cocharacters and maximal tori, or $k_s$-defined cocharacters and maximal tori,
as appropriate.
For the rest of the paper, $\KK$ denotes either of $\QQ$ or $\RR$ when the distinction is not important.
The sets $Y(\KK)$, $Y_k(\KK)$ and $Y_{k_s}(\KK)$ inherit $G$-, $G(k)$-, $G(k_s)$- and $\Gamma$-actions
from those on $Y$, $Y_k$ and $Y_{k_s}$, as appropriate, and
each element $\lambda \in Y(\KK)$ still corresponds to an R-parabolic subgroup $P_\lambda$
and an R-Levi subgroup $L_\lambda$ of $G$ (see \cite[Sec.\ 2.2]{stcc} for the case $\KK= \RR$).  If $H$ is a reductive subgroup of $G$ then we write $Y(H, \QQ)$, etc., to denote the above constructions for $H$ instead of $G$.

\subsection{$G$-varieties and cocharacter-closure}

We recall the following fundamental definition from \cite[Def.\ 1.2]{BHMR:cochars},
which extends the one given in \cite[Def.\ 3.8]{GIT}.

\begin{defn}
A subset $S$ of $V$ is said to be \emph{cocharacter-closed over $k$ (for $G$)} if for
every $v\in S$ and $\lambda \in Y_k$ such that $v':=\lim_{a\to 0} \lambda(a)\cdot v$ exists,
we have $v'\in S$.
\end{defn}

This notion is explored in detail in \cite{BHMR:cochars}.
In this section, we content ourselves with collecting some results from that paper,
together with the earlier paper \cite{GIT}.
These results, most of which are also needed in the sequel,
give a flavour of what is known about the notion of cocharacter-closure
in the case that the subset involved is a single $G(k)$-orbit.

\begin{rem}
\label{rem:HMT}
 The geometric orbit $G\cdot v$ is Zariski-closed
if and only if it is cocharacter-closed over $\ovl{k}$,
by the Hilbert-Mumford Theorem \cite[Thm.\ 1.4]{kempf}.
\end{rem}

The next result is \cite[Cor. 5.1]{BHMR:cochars}.

\begin{thm}\label{thm:mainconjcocharclosed}
Suppose $v\in V$ is such that $G(k)\cdot v$ is cocharacter-closed over $k$.
Then whenever $v'=\lim_{a\to 0} \lambda(a)\cdot v$ exists for some $\lambda \in Y_k$,
there exists $u\in R_u(P_\lambda)(k)$ such that $v' = u\cdot v$.
\end{thm}

The next result is \cite[Prop.\ 5.5]{BHMR:cochars}.

\begin{prop}\label{prop:galois_descent}
Let $v\in V$ such that $G_v(k_s)$ is $\Gamma$-stable and let $k'/k$ be
a separable algebraic extension.
If $G(k')\cdot v$ is cocharacter-closed over $k'$, then
$G(k)\cdot v$ is cocharacter-closed over $k$.
\end{prop}

The next result is \cite[Thm.\ 5.4]{BHMR:cochars}.

\begin{thm}
\label{thm:leviascentdescent}
Suppose $S$ is a $k$-defined torus of $G_v$ and set $L = C_G(S)$.
\begin{itemize}
\item[(i)] If $G(k)\cdot v$ is cocharacter-closed over $k$, then $L(k)\cdot v$ is cocharacter-closed over $k$.
\item[(ii)] If $S$ is $k$-split, then $G(k)\cdot v$ is cocharacter-closed over $k$ if and only if $L(k)\cdot v$ is cocharacter-closed over $k$.
\end{itemize}
\end{thm}

We note that, as described in the introduction, one of the main points of this paper is to show
that the converse of Proposition \ref{prop:galois_descent} holds under certain extra hypotheses,
and that the hypothesis of splitness can be removed in Theorem \ref{thm:leviascentdescent}(ii)
under the same hypotheses;
see also  \cite[Thm.\ 1.5]{BHMR:cochars}.

Our final result is a strengthening of \cite[Lem.\ 5.6]{BHMR:cochars}.

\begin{lem}
\label{lem:geom_galois_ascent}
Let $V$ be an affine $G$-variety over $k$ and let $v\in V(k)$.
Suppose there exists $\lambda\in Y_{k_s}$
such that $\lambda$ properly destabilizes $v$.
Then there exists $\mu\in Y_k$ such that $v'= \lim_{a\to 0} \mu(a)\cdot v$ exists, $v'$ is not $G(k_s)$-conjugate to $v$ and $G_v(k_s)$ normalizes $P_\mu$.  In particular, $G(k)\cdot v$ is not cocharacter-closed over $k$.
\end{lem}

\begin{proof}
 This follows from the arguments of  \cite[Lem.\ 5.6]{BHMR:cochars}.  
%[The $G_v(k_s)$-stability is not explicitly treated there, but it follows from the uniqueness of the optimal destabilizing parabolic.]
\end{proof}

\begin{rem}
\label{rem:central_destab}
 The hypotheses of Lemma~\ref{lem:geom_galois_ascent} are satisfied if $\lambda\in Y_{k_s}(Z(G^0))$ destabilizes $v$ but does not fix $v$.  For if $v':= \lim_{a\to 0} \lambda(a)\cdot v$ is $G$-conjugate to $v$ then $v'$ is $R_u(P_\lambda)$-conjugate to $v$ \cite[Thm.\ 3.3]{GIT}; but $R_u(P_\lambda)= 1$, so this cannot happen.
\end{rem}

\section{Spherical buildings and Tits' Centre Conjecture}
\label{sec:TCC}

The simplicial building $\Delta_k$ of a semisimple algebraic group $G$ over $k$ is a simplicial complex, the simplices of which correspond to the $k$-defined parabolic subgroups of $G$ ordered by reverse inclusion.  See \cite[$\S$5]{tits74} for a detailed description.  Our aim in this section is to construct for an arbitrary reductive group $G$ over $k$, objects $\Delta_k(\KK)$ for $\KK= \RR$ or $\QQ$ that correspond to the geometric realization of the spherical building of $G^0$ over $k$ (or the set of $\QQ$-points thereof) when $G^0$ is semisimple. These are slightly more general objects (possibly with a contribution from $Z(G^0)$) when $G^0$ is reductive.  Recall that $\Gamma$ denotes the Galois group of $k_s/k$.
Most of the notation and terminology below is developed in full detail in the paper \cite{stcc}---we point the reader in particular to the constructions in \cite[Sec.~2, Sec.~6.3, Sec.~6.4]{stcc}.
For the purposes of this paper, we need to extend some of the results in \emph{loc.~cit.}
(for example by incorporating the effect of the Galois group $\Gamma$),
but rather than reiterating all the details, we just gather enough
material to make our exposition here coherent.

\subsection{Definition of $\Delta_k(\KK)$}

We first form the \emph{vector building} $V_k(\KK)$ by identifying $\lambda$
in $Y_k(\KK)$ with $u\cdot \lambda$ for every $u \in R_u(P_\lambda)(k)$.
The norm function on $Y_k(\KK)$ descends to $V_k(\KK)$, because it is $G$-invariant.
This gives a well-defined function on $V_k(\KK)$, which we also call a norm,
and makes $V_k(\KK)$ into a metric space.

\begin{defn}
\label{defn:buildings}
\begin{itemize}
\item[(i)] Define $\Delta_k(\RR)$ to be the unit sphere in $V_k(\RR)$ and $\Delta_k(\QQ)$ to be the projection
of $V_k(\QQ)\setminus\{0\}$ onto $\Delta_k(\RR)$.
\item[(ii)] Two points of $\Delta_k(\KK)$ are called \emph{opposite} if they are antipodal on the sphere $\Delta_k(\RR)$.
\item[(iii)] It is clear that the conjugation action of $G(k)$ on $Y_k$ gives rise to an action of $G(k)$ on $\Delta_k(\KK)$ by isometries, and there is a natural $G(k)$-equivariant, surjective map $\zeta:Y_k(\KK)\setminus\{0\} \to \Delta_k(\KK)$.
\item[(iv)] The \emph{apartments} of $\Delta_k(\KK)$ are the sets $\Delta_{k}(T, \KK):=\zeta(Y_{k}(T, \KK))$ where $T$
runs over the maximal $k$-split tori of $G$.
\item[(v)] The metric space $\Delta_k(\KK)$ and its apartments have a simplicial structure, because any point $x= \zeta(\lambda)$ of $\Delta_k(\KK)$ gives rise to a $k$-defined parabolic subgroup $P_\lambda$ of $G^0$ (see Section~\ref{subsec:Y}); the simplicial complex consists of the proper $k$-defined parabolic subgroups of $G^0$, ordered by reverse inclusion.  
We write $\Delta_k$ for the spherical building of $G$ over $k$ 
regarded purely as a simplicial complex.  
The simplicial spherical buildings of $G^0$ and of $[G^0,G^0]$ are the same.  Our notion of opposite is compatible with the usual one for parabolic subgroups: if $\lambda\in Y(G)$ then $P_{-\lambda}$ is an opposite parabolic to $P_\lambda$.
\end{itemize}
To avoid tying ourselves in knots,
when the distinction is not important to the discussion at hand,
we loosely refer to either of the objects $\Delta_k(\QQ)$ and $\Delta_k(\RR)$ as \emph{the building of $G$ over $k$}.
\end{defn}

One can make analogous definitions of objects $\Delta_{k_s}(\KK)$ and $\Delta(\KK)=\Delta_{\ovl{k}}(\KK)$
over $k_s$ and $\ovl{k}$, respectively, with corresponding systems of apartments and maps $\zeta$.
We write $\Delta_{k_s}$ and $\Delta$ for these spherical buildings regarded as simplicial complexes. 

Because we are interested in rationality results, we need to know the relationship between $\Delta_k(\KK)$
and $\Delta_{k_s}(\KK)$.
Given a $k$-defined reductive subgroup $H$ of $G$, we also want to relate
$\Delta_{k}(H, \KK)$ to $\Delta_k(\KK)$, where $\Delta_{k}(H, \KK)$ denotes the
building of $H$ over $k$.  It is easy to see that the $\Gamma$-action on cocharacters descends (via $\zeta$) to $\Gamma$-actions by isometries on $\Delta_{k_s}(\KK)$ and $\Delta(\KK)$.

\begin{lem}
\label{lem:subbuildings}
\begin{itemize}
\item[(i)] There are naturally occurring copies of $\Delta_k(\KK)$ inside $\Delta_{k_s}(\KK)$ and $\Delta(\KK)$.
We can in fact identify $\Delta_k(\KK)$ with the set of $\Gamma$-fixed points of $\Delta_{k_s}(\KK)$.
\item[(ii)] Let $H$ be a $k$-defined reductive subgroup of $G$.
Then there is a naturally occurring copy of $\Delta_{k}(H, \KK)$ inside $\Delta_{k}(\KK)$.
\end{itemize}
\end{lem}

\begin{proof}
It is clear that $Y_k(\KK)\subseteq Y_{k_s}(\KK)\subseteq Y(\KK)$,
and $Y_k(\KK)$ is precisely the set of $\Gamma$-fixed points in $Y_{k_s}(\KK)$.
Since $R_u(P_\lambda)(k)$ acts simply transitively on
the set of $k$-defined R-Levi subgroups of $P_\lambda$, two $k$-defined cocharacters
$\lambda$ and $\mu$ are
$R_u(P_\lambda)$-conjugate if and only if they are
$R_u(P_\lambda)(k_s)$-conjugate if and only if they are $R_u(P_\lambda)(k)$-conjugate.
Following this observation through the definition of $\Delta_k(\KK)$, $\Delta_{k_s}(\KK)$ and $\Delta_{\ovl{k}}(\KK)$
is enough to prove the first assertion of (i).  It is clear that $\Delta_k(\KK)$ is fixed by $\Gamma$.  Conversely, let $x\in \Delta_{k_s}(\KK)$ be fixed by $\Gamma$.  Let $P$ be the parabolic subgroup associated to $x$.  Then $P$ is $k_s$-defined and $\Gamma$-stable, so $P$ is $k$-defined.  Pick a $k$-defined maximal torus $T$ of $P$.  There exists $\lambda\in Y_{k_s}(T,\KK)$ such that $P= P_\lambda$ \cite[8.4.4, 8.4.5]{springer}.  
Each $\gamma\in \Gamma$ maps $\lambda$ to a $R_u(P)(k_s)$-conjugate of $\lambda$. 
Now $R_u(P)$ acts simply transitively on the set of Levi subgroups of $P$, and each maximal torus of $P$ is contained in a unique Levi subgroup \cite[8.4.4]{springer}, so $R_u(P)$ acts freely on the set of maximal tori of $P$.   
But $T$ is $\Gamma$-stable, so we must have that $\Gamma$ fixes $\lambda$.  
Hence $x\in \Delta_k(\KK)$, as required.

(ii). In analogy with the first assertion of (i) (although it is slightly more subtle), the key observation is that
if $\lambda,\mu \in Y_{k}(H)$ are $R_u(P_\lambda(G))(k)$-conjugate,
then they are in fact $R_u(P_\lambda(H))(k)$-conjugate (see \cite[Lem.\ 3.3(i)]{BMRT:relative}).
Observe also that the restriction of a
$\Gamma$- and $G$-invariant norm on $Y$ to $Y(H)$
gives a $\Gamma$- and $H$-invariant norm on $Y(H)$.
\end{proof}

Henceforth, we write $\Delta_k(\KK)\subseteq \Delta_{k_s}(\KK) \subseteq \Delta(\KK)$ and
$\Delta_{k}(H, \KK)\subseteq \Delta_k(\KK)$
without any further comment.
One note of caution: the inclusion $\Delta_{k}(H, \KK)\subseteq \Delta_k(\KK)$ does not in general respect the simplicial structures on these objects.

\subsection{Convex subsets}
\label{subsec:convex}

Because any two parabolic subgroups of $G$ contain a common maximal torus,
any two points $x,y\in\Delta(\KK)$ are contained in a common apartment and, as long as these points are not
opposite each other, there is a unique geodesic $[x,y]$ joining them.
This geodesic does not depend on the apartment we find containing $x$ and $y$;
in particular, this can be done inside $\Delta_k(\KK)$ if $x,y \in \Delta_k(\KK)$ 
and inside $\Delta_k(H)$
if $x,y \in \Delta_k(H)$ for some reductive subgroup $H$ of $G$.
This leads to the following key definitions:

\begin{defn}\label{defn:convex}
\begin{itemize}
\item[(i)] A subset $\Sigma \subseteq \Delta(\KK)$ is called \emph{convex} if whenever
$x,y \in \Sigma$ are not opposite then $[x,y] \subseteq \Sigma$.
It follows from the discussion above
that $\Delta_k(\KK)$ is a convex
subset of $\Delta(\KK)$.
\item[(ii)] Given a convex subset $\Sigma$ of $\Delta(\KK)$,
its preimage $C:= \zeta^{-1}(\Sigma) \cup \{0\}$ in $Y(\KK)$ is
a union of cones $C_T := C\cap Y(T, \KK)$, where $T$ runs over the maximal tori of $G$.
The subset $\Sigma$ is called \emph{polyhedral} if each $C_T$ is a polyhedral cone and $\Sigma$ is said to have
\emph{finite type} if the set of cones $\{g\cdot C_{g\inverse Tg}\mid g\in G\}$ is finite for all $T$.
\item[(iii)] A convex subset $\Sigma$ of $\Delta(\KK)$ is called a \emph{subcomplex} if it is a union of simplices (that is, if $\lambda, \mu \in Y(\KK)$ are such that $P_\lambda = P_\mu$, then $\zeta(\lambda) \in \Sigma$ if and only if $\zeta(\mu) \in \Sigma$) and if that union of simplices forms a subcomplex in the simplicial building $\Delta$. In such a circumstance, we denote the subcomplex of $\Delta$ arising in this way by $\Sigma$ also; note that $\Sigma$ is convex in the sense of part (i) above if and only if $\Sigma$---regarded as a subcomplex of the simplicial building---is convex in the sense of simplicial buildings.
\end{itemize}
\end{defn}

The definitions above have obvious analogues for the buildings 
$\Delta_k(\KK)$ and $\Delta_{k_s}(\KK)$.

There is an addition operation on the set $V(\KK)$, given as follows.  
Let $\varphi\colon Y(\KK)\ra V(\KK)$ be the canonical projection.  
Choose a maximal torus $T$ of $G$ and $\lambda,\mu\in Y(T,\KK)$ such that 
$\varphi(\lambda)= x$ and $\varphi(\mu)= y$; 
we define $x+ y\in V(\KK)$ by $x+ y= \varphi(\lambda+ \mu)$.  
It can be shown that this does not depend on the choice of $T$; 
moreover, for any $g\in G$, $g\cdot (x+ y)= g\cdot x+ g\cdot y$.

\subsection{The destabilizing locus and complete reducibility}
\label{subsec:destabcr}

For this paper, a particularly important class of convex subsets arises from $G$-actions on affine
varieties.
Given an affine $G$-variety $V$ and a point $v \in V$, set
$$
\Sigma_v:=\{\zeta(\lambda)\mid \lambda\in Y \textrm{ and } \lim_{a\to 0}\lambda(a)\cdot v \textrm{ exists}\} \subseteq \Delta(\QQ).
$$
We call this subset the \emph{destabilizing locus} for $v$; it is a convex subset of $\Delta(\QQ)$ by \cite[Lem.\ 5.5]{stcc} (note that $\Sigma_v$ coincides with $E_{V,\{v\}}(\QQ)$ in the language of \cite{stcc}).
Similarly we write $\Sigma_{v,k}$ (resp.\ $\Sigma_{v,k_s}$) for the image in $\Delta_k(\QQ)$ 
(resp.~$\Delta_{k_s}(\QQ)$) of the $k$-defined (resp.~$k_s$-defined) characters destabilizing $v$.  
If $H$ is a reductive subgroup of $G$, then we write $\Sigma_{v,k}(H)$ for the destabilizing locus for $v$ with respect to $H$.

\begin{defn}
\label{def:cr}
A subset $\Sigma$ of $\Delta(\KK)$ is called 
\emph{completely reducible} if every point 
of $\Sigma$ has an opposite in $\Sigma$.
\end{defn}

\begin{lem}
\label{lem:cr}
Let $v\in V$.  Then:
\begin{itemize}
\item[(i)] Given $\lambda\in Y_k$ such that $\zeta(\lambda) \in \Sigma_{v,k}$, $\lambda$ has an opposite in
$\Sigma_{v,k}$ if and only if there exists $u\in R_u(P_\lambda)(k)$ such that $u\cdot\lambda$ fixes $v$,
if and only if there exists $u\in R_u(P_\lambda)(k)$ such that $\lim_{a\to 0}\lambda(a)\cdot v = u\inverse\cdot v$.
\item[(ii)] The subset $\Sigma_{v,k}$ is completely reducible if and only if $G(k)\cdot v$ is
cocharacter-closed over $k$.
\item[(iii)] The subset $\Sigma_v$ is completely reducible if and only if the orbit $G\cdot v$ is closed in $V$.
\end{itemize}
\end{lem}

\begin{proof}
We have that $\Sigma_v$ (resp.\ $\Sigma_{v,k}$) is completely reducible if and only if
for every $\lambda\in Y$ (resp.\ $\lambda\in Y_k$) such that $ \lim_{a\to 0}\lambda(a)\cdot v$ exists,
there is some $u\in R_u(P_\lambda)$ (resp.\ $u\in R_u(P_\lambda)(k)$)
such that both $u\cdot\lambda$ and
$-(u\cdot\lambda)$ destabilize $v$.
But this is true if and only if $u\cdot\lambda$ fixes $v$, which is equivalent to the fact that
$\lim_{a\to 0}\lambda(a)\cdot v = u\inverse\cdot v$, by \cite[Lem.\ 2.12]{GIT}.
This gives part (i).
Part (ii) now follows from Theorem \ref{thm:mainconjcocharclosed},
and part (iii) from Remark~\ref{rem:HMT}.
\end{proof}

\subsection{The strong Centre Conjecture and quasi-states}
\label{sec:stcc}

The aim of the paper \cite{stcc} is to study a strengthened version of Tits' Centre Conjecture for $\Delta_{k_s}(\KK)$.  Let ${\mathcal G}$ denote the group of transformations of $\Delta_{k_s}(\KK)$ generated by the isometries arising from the action of $G(k_s)$ and the action of $\Gamma$.  Note that elements of ${\mathcal G}$ map $k_s$-defined parabolic subgroups of $G$ to $k_s$-defined parabolic subgroups of $G$, so they give rise to automorphisms of the simplicial building $\Delta_{k_s}$.  Given a convex subset $\Sigma$ of $\Delta_{k_s}(\KK)$, we call a point $x\in\Sigma$ a \emph{${\mathcal G}$-centre} if it is fixed by all the elements of ${\mathcal G}$ that stabilize $\Sigma$ setwise.
We can now formulate the original Centre Conjecture in our setting.  

\begin{thm}
\label{thm:tccforsubcomplexes}
Suppose $\Sigma\subseteq\Delta_{k_s}(\KK)$ is a convex non-completely reducible subcomplex.  Then
$\Sigma$ has a ${\mathcal G}$-centre.
\end{thm}

\begin{proof}
 Theorem~\ref{thm:centreconj} asserts the existence of a stable simplex in the subcomplex
(note that the simplicial structure on $\Delta(\KK)$ does not ``see'' the difference between $G$ and $G^0$,
or between $G^0$ and its semisimple part, so the proof of the centre conjecture for subcomplexes of spherical
buildings still works for the more general class of objects we have described).
Now any element of ${\mathcal G}$ that fixes a simplex also fixes its barycentre (because the action is via isometries), and we are done.
\end{proof}

In the strong Centre Conjecture \cite[Conj.\ 2.10]{stcc}, one replaces convex non-completely reducible subcomplexes with convex non-completely reducible subsets.  Most of \cite{stcc} deals with the special case when $k= \ovl{k}$ and considers only the isometries of $\Delta_k(\KK)$ coming from the action of $G$.  We need to take the action of $\Gamma$ into account, so we briefly indicate some of the
key changes that must be made to the constructions in \cite{stcc} in order to make the results go through; see also the comments in \cite[Sec. 6.3]{stcc}.

\begin{defn}\label{defn:quasistate}
We recall the notion of a \emph{$\KK$-quasi-state} $\Xi$ from \cite[Def.\ 3.1]{stcc}: this is an assignment of a finite
set of characters $\Xi(T)$ to each maximal torus $T$ of $G$ satisfying certain conditions (see \emph{loc.~cit.}
for a precise statement).

The groups $G$ and $\Gamma$ act on quasi-states: given a $\KK$-quasi-state $\Xi$ and $g \in G$ and $\gamma \in
\Gamma$
we define new quasi-states $g_*\Xi$ and $\gamma_*\Xi$ by
$$
g_*\Xi(T):= g_!\Xi(g\inverse Tg), \qquad \gamma_*\Xi(T):=\gamma_!(\gamma\inverse\cdot T),
$$
where for a character $\chi$ of a torus $T$, $g_!\chi$ is a character of the torus $g Tg\inverse$ given by $(g_!\chi)
(gtg\inverse): = \chi(t)$ for all $t \in T$,
and similarly $\gamma_!\chi$ is a character of $\gamma\cdot T$ given by $(\gamma_!\chi)(\gamma\cdot t) := \chi(t)
$ for all $t\in T$.

We say a quasi-state is \emph{defined over a field $k'$} if it assigns $k'$-defined characters to $k'$-defined
maximal tori.
Recall also the notions of \emph{boundedness}, \emph{admissibility} 
and \emph{quasi-admissibility} for $\KK$-quasi-states, \cite[Def.\ 3.1, Def.\ 3.2]{stcc}.
\end{defn}

With these definitions in hand, we can extend \cite[Lem. 3.8]{stcc} as follows:

\begin{lem}\label{lem:galav}
Let $\Upsilon$ be a $\KK$-quasi-state which is defined over $k_s$ and define $\Xi:=\bigcup_{\gamma \in \Gamma} \gamma_*\Upsilon$ by
$\Xi(T):=\bigcup_{\gamma \in \Gamma} (\gamma_*\Upsilon)(T)$ for each maximal torus $T$ of $G$.
Then $\Xi$ is a $\KK$-quasi-state which is defined over $k_s$, and it is bounded (resp.\ quasi-admissible, admissible at $\lambda$) if $\Upsilon$ is.
Moreover, by construction, $\Xi$ is $\Gamma$-stable.
\end{lem}

\begin{proof}
There are two points which need to be made in order for the arguments already in the proof of
\cite[Lem. 3.8]{stcc} to go through.
First note that given a $k_s$-defined
maximal torus $T$ of $G$ the set of Galois conjugates of $T$ is finite (because $T$ is defined over some finite extension of $k$).
This means that, because $\Upsilon$ is $k_s$-defined, $\Xi(T)$ is still finite, so $\Xi$ is a $\KK$-quasi-state.
Now, for boundedness we need to check that if $\Upsilon$ is bounded 
then the set $\bigcup_{\gamma \in \Gamma}\left(\bigcup_{g\in G}g_*(\gamma_*\Upsilon)(T)\right)$ is finite for some (and hence all) $k_s$-defined maximal tori $T$ of $G$.
Since we can choose any $k_s$-defined maximal torus $T$, we choose one that is actually $k$-defined, and then
\begin{align*}
g_*(\gamma_*\Upsilon)(T) &= g_!(\gamma_*\Upsilon)(g\inverse Tg) \\
                                              &= g_!(\gamma_!\Upsilon(\gamma\inverse\cdot(g\inverse Tg))) \\
                                              &= g_!(\gamma_!\Upsilon((\gamma\inverse\cdot g)\inverse T(\gamma\inverse\cdot g)) \\
                                              &= \gamma_!(\gamma^\inverse\cdot g)_!\Upsilon((\gamma\inverse\cdot g)\inverse T(\gamma\inverse\cdot g)) \\
                                              &= \gamma_!((\gamma^\inverse\cdot g)_*\Upsilon)(T).
\end{align*}
Therefore, we can write
$$
\bigcup_{\gamma \in \Gamma}\left(\bigcup_{g\in G}(g_*(\gamma_*\Upsilon))(T)\right)
=
\bigcup_{\gamma \in \Gamma} \gamma_!\left(\bigcup_{g \in G}(\gamma^\inverse\cdot g)_*\Upsilon(T)\right).
$$
Now, since $\Upsilon$ is bounded, the second union on the RHS is finite for every $\gamma$, 
and because $\Upsilon$ is $k_s$-defined and the set of Galois conjugates of a $k_s$-defined character is finite, the whole RHS is finite.
This proves the boundedness assertion.  
The other assertions follow as in \emph{loc.~cit.}
\end{proof}

Using Lemma \ref{lem:galav} we can ensure that, when appropriate, the states and quasi-states constructed during the course of the paper \cite{stcc} are Galois-stable.
In particular, we get the following variant of \cite[Thm. 5.5]{stcc}:

\begin{thm}
\label{thm:singleapt}
Suppose $\Sigma \subseteq \Delta(\QQ)$ is a convex polyhedral non-completely reducible subset of finite type
contained in a single apartment of $\Delta(\QQ)$.
Then $\Sigma$ has a ${\mathcal G}$-centre.
In particular, if $\Sigma$ is stabilized by all of $\Gamma$, then there exists a $\Gamma$-fixed point in $\Sigma$.
\end{thm}

\begin{proof}
Using Lemma \ref{lem:galav}, one can ensure that the quasi-state constructed in \cite[Lem. 5.2]{stcc} which is
used in the proof of \cite[Thm. 5.5]{stcc} is also stable under the relevant elements of ${\mathcal G}$.
The proofs in \emph{loc.~cit.} now go through.
\end{proof}

\begin{rem}
 In our application of Theorem~\ref{thm:singleapt} to the proof of Theorem~\ref{thm:nice_hyps} below, $\Sigma$ is a $\Gamma$-stable subset of $\Delta_{k_s}(\QQ)$ and we want to show that $\Sigma$ has a $\Gamma$-fixed point.  A striking feature of Theorem~\ref{thm:singleapt} is that we do {\bf not} require the apartment containing $\Sigma$ to be $k_s$-defined: it can be any apartment of the building $\Delta(\KK)= \Delta_{\ovl{k}}(\KK)$.
\end{rem}

We note here that, unfortunately, we do not know \emph{a priori} that any cocharacter corresponding to the fixed point given by Theorem~\ref{thm:singleapt} properly destabilizes $v$.  Moreover, we may need to consider cocharacters of $Z(G^0)$, which do not correspond to simplices of the spherical building at all.  These technical issues are at the heart of many of the complications in the proofs in Section~\ref{sec:tccproofs} below.

\section{Proofs of the main results}
\label{sec:tccproofs}

Having put in place the building-theoretic technology needed for our proofs in the previous section,
we start this section with a few more technical results to be used for the main theorems.
As always, $V$ denotes a $k$-defined affine $G$-variety, and $v\in V$.  One obstacle to proving Theorems~\ref{thm:subcx} and \ref{thm:nice_hyps} is that we need to deal with cocharacters that live in $Z(G^0)$, which are not detected by the simplicial building (cf.\ the proof of Theorem~\ref{thm:tccforsubcomplexes}).  An extra problem for Theorem~\ref{thm:nice_hyps} is that we need to factor out some simple components of $G^0$.  The following results let us deal with these difficulties.

Let $N$ be a product of certain simple factors of $G^0$, and let $S$ be a torus of $Z(G^0)$.
Let $M$ be the product of the remaining simple factors of $G^0$
together with $Z(G^0)$.  Suppose that $N$ and $S$ are normal in $G$ (this implies that $M$ is normal in $G$ as well), and that $N$ and $S$ both fix $v$.  Set $G_1= G/NS$ and let $\pi\colon G\ra G_1$ be the canonical projection.  Since $Z(G^0)$ is $k_s$-defined and $k_s$-split, $S$ is $k_s$-defined, and it is clear that $N$ is $k_s$-defined.  So $G_1$ and $\pi$ are $k_s$-defined.  We have a $k_s$-defined action of $G_1$ on the fixed point set $V^{NS}$ (note that $V^{NS}$ is $G$-stable). 

\begin{lem}
\label{lem:normal}
\begin{itemize}
\item[(i)] For any $\mu_1\in Y_{k_s}(G_1)$, there exist $n\in {\mathbb N}$ and $\mu\in Y_{k_s}(M)$ such that $\pi\circ \mu= n\mu_1$.
\item[(ii)] Let $\lambda\in Y_{k_s}$.
Then $\lambda$ destabilizes $v$ over $k_s$ in $V$ if and only if $\pi\circ\lambda$
destabilizes $v$ over $k_s$ in $V^{NS}$.
Moreover, if this is the case then
$\lim_{a\to 0}\lambda(a)\cdot v=\lim_{a\to 0}(\pi\circ \lambda)(a)\cdot v$
belongs to $R_u(P_\lambda(G))(k_s)\cdot v$ if and only if it belongs to
$R_u(P_{\pi\circ \lambda}(G_1))(k_s)\cdot v$.
\item[(iii)] $G_1(k_s) \cdot v$ is cocharacter-closed over $k_s$ if and only if $G(k_s)\cdot v$ is cocharacter-closed over $k_s$.
\end{itemize}
\end{lem}

\begin{proof}
(i). Let $\mu_1 \in Y_{k_s}(G_1)$. Since $\mu_1$ is $k_s$-defined,
we can choose a $k_s$-defined maximal torus $T_1 \subseteq G_1$
with $\mu_1 \in Y_{k_s}(T_1)$.
Since $\pi$ is separable and $k_s$-defined, $\pi^{-1}(T_1) \subseteq G$ is $k_s$-defined \cite[Cor.\ 11.2.14]{springer}.
Hence $\pi^{-1}(T_1)$ has a $k_s$-defined maximal torus $T$.
Let $T' = T \cap M$, a $\overline{k}$-defined torus of $M$.
Now $\pi(T') = \pi(T)$ is a maximal torus of $G_1$ by \cite[Prop.\ 11.14(1)]{borel};
but $\pi(T')$ is contained in $T_1$, so we must have $\pi(T') = T_1$.   The surjection $T' \rightarrow T_1$ induces a surjection $\QQ \otimes_\ZZ Y_{\overline{k}}(T') \rightarrow \QQ \otimes_\ZZ Y_{\overline{k}}(T_1)$ (the map before tensoring maps onto a finite-index subgroup: e.g., by transposing the injective map on character groups \cite[Thm.\ 7.3]{Water}),
hence there exist $n\in {\mathbb N}$ and $\mu\in Y_{\overline{k}}(T')$ such that $\pi\circ \mu= n\mu_1$.  As $\mu \in Y_{\overline{k}}(T) = Y_{k_s}(T)$, $\mu$ is $k_s$-defined
as required.

(ii). The first assertion is immediate, as is the assertion that the limits coincide.  Since $\pi$ is an epimomorphism, 
we have $\pi(R_u(P_\lambda(G))) = R_u(P_{\pi\circ\lambda}(G_1))$ (see \cite[Cor.\ 2.1.9]{CGP}).  
Moreover, since $\lambda$ normalizes $NS$, the restriction of $\pi$ to $R_u(P_\lambda(G))$ 
is separable (see \cite[Prop.\ 2.1.8(3) and Rem.\ 2.1.11]{CGP}) and $k_s$-defined, and hence 
is surjective on $k_s$-points (cf.\ \cite[Cor.\ 18.5]{Water}).  This implies that if the common 
limit $v'$ is in $R_u(P_{\pi\circ\lambda}(G_1))(k_s)\cdot v$, it is contained in $R_u(P_\lambda(G))(k_s) \cdot v$.  
The reverse implication is clear, since $\pi$ is $k_s$-defined.

(iii). Suppose $G(k_s)\cdot v$ is not cocharacter-closed over $k_s$.
Then there exists $\lambda\in Y_{k_s}$ such that $\lim_{a\to 0}\lambda(a)\cdot v$ exists but does not belong to $R_u(P_\lambda(G))(k_s)\cdot v$.  Then $\lim_{a\to 0}(\pi\circ \lambda)(a)\cdot v$ exists but does not belong to $R_u(P_\lambda(G_1))(k_s)\cdot v$, by part (ii).  Hence $G_1(k_s)\cdot v$ is not cocharacter-closed over $k_s$, by Theorem ~\ref{thm:mainconjcocharclosed}.

Now suppose $G_1(k_s)\cdot v$ is not cocharacter-closed over $k_s$.
Then there exists $\mu_1\in Y_{k_s}(G_1)$ such that $v': = \lim_{a\to 0} \mu_1(a)\cdot v$ exists and does not belong to $R_u(P_\mu(G_1))(k_s)\cdot v$.  Replacing $\mu_1$ with a positive multiple $n\mu_1$ of $\mu_1$ if necessary, it follows from part (i) that there exists $\mu\in Y_{k_s}$ such that $\pi\circ \mu= \mu_1$.  Then $\lim_{a\to 0} \mu(a)\cdot v$ is equal to $v'$ and $v'$ does not belong to $R_u(P_\lambda(G))(k_s)\cdot v$, by part (ii).  Hence $G(k_s)\cdot v$ is not cocharacter-closed over $k_s$, by Theorem ~\ref{thm:mainconjcocharclosed}.
\end{proof}

\begin{rem}
 We insist in Lemma~\ref{lem:normal}(i) that $\lambda$ be a cocharacter of $M$ because we need this in the proof of Theorem~\ref{thm:nice_hyps}.
\end{rem}

\begin{lem}
\label{lem:L_subcx}
 Let $G$ be connected, let $S$ be a $k_s$-torus of $G_v$ and set $L= C_G(S)$.
Suppose that for every $\lambda, \mu\in Y_{k_s}$ such that $P_\lambda= P_\mu$,
either both of $\lambda$ and $\mu$ destabilize $v$ or neither does.
Then for every $\lambda, \mu\in Y_{k_s}(L)$ such that $P_\lambda(L)= P_\mu(L)$,
either both of $\lambda$ and $\mu$ destabilize $v$ or neither does.
\end{lem}

\begin{proof}
 Let $\lambda, \mu\in Y_{k}(L)$ such that $P_\lambda(L)= P_\mu(L)$.  We can choose $\sigma \in Y_k(S)$ such that $L_\sigma= L$ \cite[Lem.\ 2.5]{BHMR:cochars}.  Then $P_\lambda(L_\sigma) = P_\mu(L_\sigma)$ and there exists $n\in {\mathbb N}$ such that $P_{n\sigma+\lambda} = P_\lambda(L_\sigma)R_u(P_\sigma)$ and $P_{n\sigma+\mu}= P_\mu(L_\sigma)R_u(P_\sigma)$ \cite[Lem.\ 6.2(i)]{BMR}.  By hypothesis, either both of $n\sigma+\lambda$ and $n\sigma+\mu$ destabilize $v$, or neither one does.  In the first case, since $\sigma$ fixes $v$, both $\lambda$ and $\mu$ must destabilize $v$.  Conversely, in the second case neither $\lambda$ nor $\mu$ can destabilize $v$.
\end{proof}

\begin{lem}
\label{lem:sameapt}
 Let $T$ be a maximal torus of $G$, let $\mu_1,\ldots, \mu_r\in Y(T)\backslash \{0\}$, let $\mu= \sum_{i= 1}^r \mu_i$ and assume $\mu\neq 0$.  Suppose $g\in G$ and $g\cdot \zeta(\mu_i)= \zeta(\mu_i)$ for all $1\leq i\leq r$.  Then $g\cdot \zeta(\mu)= \zeta(\mu)$.
\end{lem}

\begin{proof}
Clearly, there is a permutation $\tau\in S_r$ such that none of the sums 
$\sum_{i= 1}^t \mu_{\tau(i)}$ is 0 for $1\leq t\leq r$.  
Consider the special case $r= 2$ (the general case follows easily by induction on $r$).  
Recall the addition operation $+$ on $V(\KK)$ and the canonical projection 
$\varphi\colon Y(\KK)\ra V(\KK)$ from Section~\ref{subsec:convex}.  
Let $\xi\colon V(\KK)\backslash \{0\}\ra \Delta(\KK)$ be the canonical projection.  
Note that $\varphi$ and $\xi$ are $G$-equivariant.  
Moreover, as $g$ fixes $\zeta(\mu_1)$ and $g$ acts as an isometry, 
$g$ fixes $\varphi(\mu_1)$, and likewise $g$ fixes $\varphi(\mu_2)$.  We have
\begin{align*}
g\cdot \zeta(\mu) & = g\cdot \zeta(\mu_1+ \mu_2)= g\cdot \xi(\varphi(\mu_1+ \mu_2))= \xi(g\cdot \varphi(\mu_1+ \mu_2))= \xi(g\cdot (\varphi(\mu_1)+ \varphi(\mu_2))) \\
&  = \xi(g\cdot \varphi(\mu_1)+ g\cdot \varphi(\mu_2))= \xi(\varphi(\mu_1)+ \varphi(\mu_2))= \xi(\varphi(\mu_1+ \mu_2))= \zeta(\mu_1+ \mu_2)= \zeta(\mu), 
\end{align*}
 as required.
\end{proof}

We now have everything in place to prove Theorem~\ref{thm:subcx}.

\begin{proof}[Proof of Theorem~\ref{thm:subcx}]
 For part (i), suppose $v\in V(k)$ and $G(k_s)\cdot v$ is not cocharacter-closed over $k_s$.
Clearly there is no harm in assuming $S$ is a maximal $k$-defined torus of $G_v$, so we shall do this.
Since the closed subgroup $\overline{G_v(k_s)}$ generated by $G_v(k_s)$ is $k_s$-defined and $\Gamma$-stable, it is $k$-defined.
Hence $S$ is a maximal torus of $\overline{G_v(k_s)}$; 
in particular, $S$ is a maximal $k_s$-defined torus of $G_v$.
Set $H= C_G(S)$.  If $\sigma\in Y_k(H)$ and $\sigma$ destabilizes $v$ but does not fix $v$ then $\sigma$ properly destabilizes $v$ over $k_s$ for $G$, by \cite[Lem.\ 2.8]{BHMR:cochars}.  Hence it is enough to prove that such a $\sigma$ exists.
 
By Theorem \ref{thm:leviascentdescent}(ii), $H(k_s)\cdot v$ is not cocharacter-closed over $k_s$.  So we can choose $\mu\in Y_{k_s}(H)$ such that $\mu$ properly destabilizes $v$ over $k_s$ for $H$.
If $\mu\in Y_{k_s}(Z(H^0))$ then we are done by Lemma~\ref{lem:geom_galois_ascent}
and Remark~\ref{rem:central_destab}.
So assume otherwise.  Then $P_\mu(H^0)$ is a proper subgroup of $H^0$.  By Lemma~\ref{lem:L_subcx}, $\Sigma_{v,k_s}(H)$ is a subcomplex of $\Delta_{k_s}(H,\KK)$.  It follows from Lemma~\ref{lem:cr} that $\Sigma_{v,k_s}(H)$ 
is not completely reducible, since if $Q$ is an opposite parabolic to $P_\mu(H^0)$ in $H^0$ then there exists $\mu'\in Y_{k_s}(H)$ such that $P_{\mu'}= Q$ and $\mu'$ is opposite to $\mu$, which is impossible.  Clearly, $\Sigma_{v,k_s}(H)$ is $\Gamma$- and $H_v(k_s)$-stable,
so by Theorem~\ref{thm:tccforsubcomplexes} there is a $\Gamma$- and $H_v(k_s)$-fixed simplex $s \in \Sigma_{v,k_s(H)}$,
corresponding to some proper parabolic subgroup $P$ of $H^0$.
There exists $\sigma\in Y_k(H)$ such that $P= P_\sigma(H^0)$ \cite[Lem.\ 2.5(ii)]{GIT}, 
and $\sigma$ destabilizes $v$ by construction.
Now $\sigma\not\in Y_{k_s}(Z(H^0))$ since $P$ is proper.
But every $k_s$-defined torus of $H_v(k_s)$ is contained in $Z(H^0)$ (since $S$ is contained in $Z(H^0)$), so $\sigma$ does not fix $v$.
As $\sigma$ commutes with $S$, it follows from \cite[Lem.\ 2.8]{BHMR:cochars} that $v':= \lim_{a\to 0} \sigma(a)\cdot v$ does not lie in $H(k_s)\cdot v$.
This completes the proof of (i).

For part (ii), Proposition \ref{prop:galois_descent} shows that if $G(k')\cdot v$ is
cocharacter-closed over $k'$ then $G(k)\cdot v$ is cocharacter-closed over $k$.
For the other direction, suppose that $G(k') \cdot v$ is not cocharacter-closed over $k'$.
Again by Proposition \ref{prop:galois_descent}, we may assume $k' = k_s$.
Applying part (i) with $S=1$, we find $\sigma \in Y_k$ such that $\sigma$ properly destabilizes
$v$ over $k_s$. In particular, $G(k) \cdot v$ is not cocharacter-closed over $k$.
This finishes part (ii).

Part (iii) of Theorem \ref{thm:subcx} follows using similar arguments to those in the proof of \cite[Thm.\ 5.7(ii)]{BHMR:cochars}.  Let $S$ be a $k$-defined torus of $G_v$ and let $L= C_G(S)$.
First, by the argument of \cite[Lem.\ 6.2]{BHMR:cochars}, we can assume without loss that $v\in V(k_s)$ without changing $\Sigma_{v,k_s}$.  Second, by \cite[Lem.\ 6.3]{BHMR:cochars} and the argument of the proof of \cite[Thm.\ 6.1]{BHMR:cochars}, we can pass to a suitable $G$-variety $W$ and find $w\in W(k)$ such that $\Sigma_{w,k_s}= \bigcap_{\gamma\in \Gamma} \gamma\cdot \Sigma_{v,k_s}$; in particular, $\Sigma_{w,k_s}$ is a subcomplex of $\Delta_k(\QQ)$ and $\Sigma_{w,k}= \Sigma_{v,k}$.  The arguments of \cite[Sec.\ 6]{BHMR:cochars} also show that $S$ fixes $w$.  Hence we can assume without loss that $v\in V(k)$.  As before, Lemma~\ref{lem:L_subcx} implies that $\Sigma_{v,k_s}(L)$ is a subcomplex of $\Delta_{k_s}(L,\KK)$.
We may thus apply part (ii) and assume $k=k_s$.
But then $S$ is $k$-split, so the result follows
from Theorem \ref{thm:leviascentdescent}. 
\end{proof}

\begin{rem}
 We do not know how to prove that $P_\sigma(G^0)$ from Theorem~\ref{thm:subcx}(i) is normalized by $G_v(k_s)$, but the proof does show that $P_\sigma(G^0)$ is normalized by $H_v(k_s)$.
\end{rem}

We finish the section by proving Theorem \ref{thm:nice_hyps}.

\begin{proof}[Proof of Theorem~\ref{thm:nice_hyps}]
 For part (i), suppose $v\in V(k)$ and $G(k_s)\cdot v$ is not cocharacter-closed over $k_s$.  Recall that a connected algebraic group is nilpotent if and only if it contains just one maximal torus (see \cite[21.4 Prop.\ B]{Hum}).  Let $G_i$ be a simple component of $G^0$.  If ${\rm rank}(G_i)= 1$ and $\dim (G_i)_v \geq 2$ then $(G_i)_v^0$ must contain a Borel subgroup $B_i$ of $G_i$: but then the orbit map $G_i\ra G_i\cdot v$ factors through the connected projective variety $G_i/B_i$ and hence is constant, so $(G_i)_v= G_i$.
 
 Let $N$ be the product of the simple components of $G^0$ that fix $v$, and let $K$ be the product of the remaining simple components of $G^0$ together with $Z(G^0)^0$.  Then $N$ and $K$ are $\Gamma$-stable, so they are $k$-defined. The next step is to factor out $N$ and reduce to the case when the stabilizer has nilpotent identity component.   As in the proof of Theorem~\ref{thm:subcx}, $\ovl{G_v(k_s)}$ is $k$-defined and we may assume $S$ is a maximal $k$-defined torus of $G_v$ and a $k_s$-defined maximal torus of $\overline{G_v(k_s)}$.  We can choose $k$-defined tori $S_0$ of $K$ and $S_2$ of $N$ such that $S= S_0S_2$.  Note that $K_v^0$ is nilpotent---this holds by assumption in case (a), and by the above argument in case (b)---so $S_0$ is the unique maximal $k$-defined torus of $\ovl{K_v(k_s)}$.
 
 Let $H_0= N_G(S_0)$, let $H= N_{H_0}(N)$ and let $M= C_{K}(S_0)$.  Then $H_0$ is $k$-defined \cite[Lem.\ A.2.9]{CGP}, so $H$ is $k$-defined as it is $\Gamma$-stable and has finite index in $H^0$.  Note that $H^0= NM= C_G(S_0)^0$, so $N$ is a product of simple components of $H^0$ and $M$ is the product of $Z(H^0)^0= S^0Z(G^0)^0$ with the remaining simple components of $H^0$.  The subgroup $M$ of $H$ is normal, so it is $\Gamma$-stable and hence $k$-defined.  Now $M_v^0$ is nilpotent since $K_v^0$ is, so $M_v^0$ has a unique maximal torus $S'$---in particular, $S_0\subseteq S'$ and $S_0$ is the unique maximal torus of $\overline{M_v(k_s)}$.  Since $G_v(k_s)$ normalizes $N$ and $K$, $G_v(k_s)$ normalizes $N$ and $S_0$, so $G_v(k_s)\subseteq H$; it follows that $H_v(k_s)= G_v(k_s)$.
  
 Let $H_1= H/NS_0$ and let $\pi\colon H\ra H_1$ be the canonical projection.  We wish to find $\lambda_1\in Y_{k_s}(H_1)$ such that $\lambda_1$ properly destabilizes $v$ over $k_s$ and $P_{\lambda_1}(H_1)$ is $k$-defined.
Note that no nontrivial $k_s$-defined cocharacter of $H_1$ fixes $v$; 
for if $0\neq \lambda_1\in Y_{k_s}(H_1)$ fixes $v$ then by Lemma~\ref{lem:normal}, 
there exist $n\in {\mathbb N}$ and $\lambda\in Y_{k_s}(M)$ such that 
$\pi\circ \lambda= n\lambda_1$, and $\langle {\rm Im}(\lambda)\cup S\rangle$ 
is a $k_s$-defined torus of $\ovl{G_v(k_s)}$ that properly contains $S$, contradicting the maximality of $S$.  
Clearly, $(H_1)_v^0$ is nilpotent with unique maximal torus $S_1':= \pi(S')$.    
Since $H^0= C_G(S_0)^0$, $H(k_s)\cdot v$ is not cocharacter-closed over $k_s$, 
by Theorem~\ref{thm:leviascentdescent}(ii) and \cite[Cor.\ 5.3]{BHMR:cochars}.  
Hence $H_1(k_s)\cdot v$ is not cocharacter-closed over $k_s$ (Lemma~\ref{lem:normal}).  
Let $\lambda_1\in Y_{k_s}(H_1)$ such that $\lambda_1$ destabilizes $v$.  By Lemmas~\ref{lem:geom_galois_ascent} and \ref{lem:normal}, 
we can assume $\lambda_1$ does not properly destabilize $v$ over $\ovl{k}$.  
Therefore, there exists $u\in R_u(P_{\lambda_1}(H_1))$ such that $u\cdot \lambda_1$ fixes $v$; 
then $u\cdot\lambda_1$ must be a cocharacter of $S_1'$.  
It follows that $\Sigma_{v,k_s}(H_1)\subseteq \Delta_{k_s}(T_1, \QQ)$, 
where $T_1$ is a fixed maximal torus of $H_1$ that contains $S_1'$.  Note that $T_1$ and $S_1'$ need not be $k$-defined, or even $k_s$-defined.
As $H_1(k_s)\cdot v$ is not cocharacter-closed over $k_s$, $\Sigma_{v,k_s}(H_1)$ 
is not completely reducible (Lemma~\ref{lem:cr}(ii)).  Now $\Sigma_{v,k_s}(H_1)$ 
is stabilized by $\Gamma$ and by $(H_1)_v(k_s)$, so it follows from Theorem~\ref{thm:singleapt} 
that $\Sigma_{v,k_s}(H_1)$ contains a $\Gamma$-fixed and $(H_1)_v(k_s)$-fixed point $x_1$.  
We can write $x_1= \zeta(\mu_1)$ for some $\mu_1\in Y_{k_s}(H_1)$.  
Then $\mu_1$ destabilizes $v$ but does not fix $v$; moreover, $P_{\mu_1}(H_1^0)$ 
is $\Gamma$-stable and is normalized by $(H_1)_v(k_s)$.  In particular, $P_{\mu_1}(H_1^0)$ is $k$-defined.
 
 By Lemma~\ref{lem:normal}, there exist $n\in {\mathbb N}$ and $\mu\in Y_{k_s}(M)$ such that $\pi\circ \mu= n\mu_1$ and $\mu$ destabilizes $v$; note that $\mu$ does not fix $v$, because $\mu_1$ does not.  The map $\pi$ gives a bijection between the parabolic subgroups of $M^0$ and the parabolic subgroups of $H_1^0$.  So $P_\mu(M^0)$ is $\Gamma$-stable---because $P_{\mu_1}(H_1^0)$ is---and hence is defined over $k$.  As $\pi(H_v(k_s))$ is contained in $(H_1)_v(k_s)$ and $(H_1)_v(k_s)$ normalizes $P_{\mu_1}(H_1^0)$, $H_v(k_s)$ normalizes $P_\mu(M^0)$.
 
 After replacing $\mu$ if necessary with an $R_u(P_\mu(H^0))(k_s)$-conjugate of $\mu$, we can assume that $\mu$ is a cocharacter of a $k$-defined maximal torus $T$ of $P_\mu(H^0)$.  Let $\mu^{(1)},\ldots, \mu^{(r)}$ be the $\Gamma$-conjugates of $\mu$.  These are cocharacters of $T$, so they all commute with each other.  Set $\sigma= \sum_{i= 1}^r \mu^{(i)}$, a $k$-defined cocharacter of $T$.  Note that $\sigma$ centralizes $S= S_0S_2$.  Now $\pi\circ \sigma$ destabilizes $v$ but does not fix $v$ (since $\pi\circ \sigma\neq 0$), so $\sigma$ does not fix $v$.  This implies by \cite[Lem.\ 2.8]{BHMR:cochars} that $\sigma$ properly destabilizes $v$ over $k_s$ for $G$.  Since $H_v(k_s)$ is $\Gamma$-stable and fixes $\zeta(\mu)$, $H_v(k_s)$ fixes $\zeta(\mu^{(i)})$ for all $1\leq i\leq r$.  It follows from Lemma~\ref{lem:sameapt} that $H_v(k_s)$ fixes $\zeta(\sigma)$: that is, for all $h\in H_v(k_s)$, there exists $u\in R_u(P_\sigma(H^0))(k_s)$ such that $h\cdot \sigma= u\cdot \sigma$.  Hence $P_\sigma(G^0)$ is normalized by $H_v(k_s)= G_v(k_s)$.  This completes the proof of (i).
 
 Parts (ii) and (iii) now follow as in the proof of Theorem~\ref{thm:subcx} (there is no need to reduce to the case when $v$ is a $k$-point in (iii) because we already assume this).
\end{proof}

\begin{rem}
\label{rem:arbpt}
 It can be shown that Theorem~\ref{thm:nice_hyps}(iii) actually holds without the assumption that $v$ is a $k$-point.  Here is a sketch of the proof.  It is enough to prove that Levi ascent holds.  Without loss, assume $S$ is a maximal $k$-defined torus of $G_v$.  As in the proof of Theorem~\ref{thm:subcx}(iii), we replace $v$ with a $k$-point $w$ of a $k$-defined $G$-variety $W$, with the property that $\Sigma_{w,k_s}\subseteq \Sigma_{v,k_s}$ and $\Sigma_{w,k_s}= \Sigma_{v,k_s}$.  By the arguments of \cite[Sec.\ 6]{BHMR:cochars}, we can assume that $S$ and $N$ fix $w$.  Suppose $G(k)\cdot v$ is not cocharacter-closed over $k$.  Then $G(k)\cdot w$ is not cocharacter-closed over $k$, so $L(k_s)\cdot w$ is not cocharacter-closed over $k_s$, by Galois descent and split Levi ascent.  It follows that $H_1(k_s)\cdot w$ is not cocharacter-closed over $k_s$, where $H_1$ is defined as in the proof of Theorem~\ref{thm:nice_hyps}.  We do not know whether hypotheses (a) and (b) of Theorem~\ref{thm:nice_hyps} hold for $w$.  The key point, however, that makes the proof of Theorem~\ref{thm:nice_hyps}(i) work is that $\Sigma_{v,k_s}(H_1)$ is contained in a single apartment of $\Delta_{k_s}(H_1,\QQ)$.  The analogous property holds for $\Sigma_{w,k_s}(H_1)$ since $\Sigma_{w,k_s}(H_1)\subseteq \Sigma_{v,k_s}(H_1)$.  Hence Galois ascent holds and $H_1(k)\cdot w$ is not cocharacter-closed over $k$.  Then $H_1(k)\cdot v$ is not cocharacter-closed over $k$, and the result follows.
\end{rem}

\section{Applications to $G$-complete reducibility}

Many of the constructions in this paper, and in the key references 
\cite{GIT}, \cite{stcc}, \cite{BHMR:cochars},
were inspired originally by the study of Serre's notion of \emph{$G$-complete reducibility} for subgroups of $G$.
We refer the reader to \cite{serre2} and \cite{BMR} for a thorough introduction to the theory.
We simply record the basic definition here:

\begin{defn}\label{def:gcr}
A subgroup $H$ of $G$ is said to be \emph{$G$-completely reducible over $k$} if whenever $H$ is contained
in a $k$-defined R-parabolic subgroup $P$ of $G$, there exists a $k$-defined R-Levi subgroup $L$ of $P$
containing $H$.  (We do not assume that $H$ is $k$-defined.)
\end{defn}

The following is \cite[Thm.\ 9.3]{BHMR:cochars}.

\begin{thm}
\label{thm:crvscocharclosed}
Let $H$ be a subgroup of $G$ and let $\tuple{h}\in H^n$
be a generic tuple of $H$ (see \cite[Def.\ 5.4]{GIT}).
Then $H$ is $G$-completely reducible over $k$ if and only if
$G(k) \cdot \tuple{h}$ is cocharacter-closed over $k$, where $G$ acts on
$G^n$ by simultaneous conjugation.
\end{thm}

Theorem~\ref{thm:crvscocharclosed} allows us to prove results
about $G$-complete reducibility over $k$ using our results on
geometric invariant theory.
If $G$ is connected, and $\tuple{h}\in G^n$
is a generic tuple for a subgroup $H$ of $G$, then $\Sigma_\tuple{h}$
is a subcomplex of $\Delta_G(\QQ)$, since for any $\lambda\in Y$,
$\lambda$ destabilizes $\tuple{h}$ if and only if $H\subseteq P_\lambda$;
this means that we are %apparently 
in the territory of Theorem \ref{thm:subcx}.

\begin{proof}[Proof of Theorem \ref{thm:LeviGaloisdescent-ascentforGcr}]
Let $\tuple{h}$ be a generic tuple of $H$.
Then $\Sigma_{\tuple{h},k_s}$ is a subcomplex of $\Delta_{G,k_s}$,
and $C_G(H) = G_{\tuple{h}}$.
The result now follows from Theorems \ref{thm:crvscocharclosed}
and \ref{thm:subcx}(iii).
\end{proof}

This theory has a counterpart for Lie subalgebras of $\gg:= {\rm Lie}(G)$.  
The basic definitions and results were covered for algebraically closed 
fields in \cite{mcninch} and \cite[Sec.\ 3.3]{BMRT:relative}, 
but the extension to arbitrary fields is straightforward 
(cf.\ \cite[Rem.\ 4.16]{BMRT:relative}).

\begin{defn}
 A Lie subalgebra $\hh$ of $\gg$ is \emph{$G$-completely reducible over $k$} 
if whenever $P$ is a $k$-defined parabolic subgroup of $G$ such that 
$\hh\subseteq \Lie(P)$, there exists a $k$-defined Levi subgroup 
$L$ of $P$ such that $\hh\subseteq \Lie(L)$.  
(We do not assume that $\hh$ is $k$-defined.)
\end{defn}

The group $G$ acts on $\gg^n$ via the simultaneous adjoint action 
for any $n\in \NN$.  
The next result follows from \cite[Lem.\ 3.8]{BMRT:relative} 
and the arguments in the proofs of \cite[Thm.\ 4.12(iii)]{BMRT:relative} 
(cf.\ \cite[Thm.\ 3.10(iii)]{BMRT:relative}).

\begin{thm}
\label{thm:crvscocharclosedLie}
 Let $\hh$ be a Lie subalgebra of $\gg$ and let $\tuple{h}\in \hh^n$ 
such that the components of $\tuple{h}$ generate $\hh$ as a Lie algebra.  
Then $\hh$ is $G$-completely reducible over $k$ if and only if 
$G(k)\cdot \tuple{h}$ is cocharacter-closed over $k$.
\end{thm}

We now give the applications of our earlier results to $G$-complete reducibility over $k$ for Lie algebras.

\begin{thm}
 Let $\hh$ be a Lie subalgebra of $\gg$.
 \begin{itemize}
  \item[(i)] Suppose $\hh$ is $k$-defined, and let $k'/k$ be a separable algebraic extension.  
Then $\hh$ is $G$-completely reducible over $k'$ if and only if 
$\hh$ is $G$-completely reducible over $k$.
  \item[(ii)] Let $S$ be a $k$-defined torus of $C_G(\hh)$ and set $L= C_G(S)$. 
Then $\hh$ is $G$-completely reducible over $k$ if and only if $\hh$ is $L$-completely reducible over $k$.
 \end{itemize}
\end{thm}

\begin{proof}
 Pick $\tuple{h}\in \hh^n$ for some $n\in \NN$ such that the components of 
$\tuple{h}$ generate $\hh$ as a Lie algebra.  
If $\hh$ is $k$-defined then we can assume that 
$\tuple{h}\in \hh(k)^n$.  
Part (i) now follows from 
Theorems~\ref{thm:crvscocharclosedLie} and \ref{thm:subcx}(ii), 
and part (ii) from Theorems~\ref{thm:crvscocharclosedLie} and \ref{thm:subcx}(iii).
\end{proof}

%%%%%%%%%%%%%%%%%%%%%%%%%%%%%%%%%%%%%%%%%%%%%%%%%%%%%%%%%%%%%%%%%%%%%%
%%%%%%%%%%%%% Acknowledgments
%%%%%%%%%%%%%%%%%%%%%%%%%%%%%%%%%%%%%%%%%%%%%%%%%%%%%%%%%%%%%%%%%%%%%%

\bigskip
{\bf Acknowledgments}:
The authors acknowledge the financial support of EPSRC Grant EP/L005328/1 and
of Marsden Grants UOC1009 and UOA1021.
Part of the research for this paper was carried out while the
authors were staying at the Mathematical Research Institute
Oberwolfach supported by the ``Research in Pairs'' programme.
Also, parts of this paper were written during mutual visits to Auckland,
Bochum and York.  We thank the referee for a number of suggestions to improve the exposition.

%%%%%%%%%%%%%%%%%%%%%%%%%%%%%%%%%%%%%%%%%%%%%%%%%%%%%%%%%%%%%%%%%%%%%%
%%%%%%%%%%%%% bibliography
%%%%%%%%%%%%%%%%%%%%%%%%%%%%%%%%%%%%%%%%%%%%%%%%%%%%%%%%%%%%%%%%%%%%%%

\end{document}